\documentclass[10pt]{amsart}

\usepackage{amsmath}
\usepackage{amssymb}
\usepackage{amsthm}
\usepackage{textcomp}
\usepackage{indentfirst}
\usepackage{mathrsfs}
\usepackage{tikz-cd}
\usepackage{epsfig}
\usepackage{color}
\usepackage{bbm}
\usepackage[yyyymmdd,hhmmss]{datetime}
\usepackage{thmtools}
\usepackage{mathdots}
\declaretheoremstyle[headfont=\normalfont]{normalhead}

\usepackage[margin=1.175in]{geometry}
\usepackage{lipsum}

\usepackage[pagebackref]{hyperref}
\hypersetup{
    colorlinks = true,
    citecolor = [rgb]{0, 0.5, 0.00},
}

\renewcommand*{\backref}[1]{}
\renewcommand*{\backrefalt}[4]{%
    \ifcase #1 (Not cited.)%
    \or        (Cited on page~#2.)%
    \else      (Cited on pages~#2.)%
    \fi}

\newcommand{\U}{\mathbf{U}}
\newcommand{\UU}{\operatorname{U}}
\renewcommand{\u}{\operatorname{U}}

\newcommand\C{\mathbb{C}}

\newcommand\GL{\mathrm{GL}}

\newcommand\calO{\mathcal{O}}
\newcommand\frakp{\mathfrak{p}}

\newcommand{\Sh}{\mathrm{Sh}}

\newcommand{\SO}{\operatorname{SO}}
\newcommand{\Sp}{\operatorname{Sp}}

\newcommand{\cind}{\mathrm{cInd}}

\renewcommand{\bar}[1]{\sigma(#1)}

\newcommand{\vol}{\mathrm{vol}}
\newcommand{\volx}{\mathrm{vol}^\times}
\newcommand{\tauF}{\tau|_{F^\times}}

\newtheorem{thm}{Theorem}[section]
\newtheorem{lem}[thm]{Lemma}
\newtheorem{prop}[thm]{Proposition}

\newcommand{\pmat}[4]{
\begin{pmatrix}
#1 & #2\\
#3 & #4
\end{pmatrix}
}

\title{On the Rankin-Selberg gamma factor of simple supercuspidal representations of the unitary group for $p$-adic local fields}
\author{Philip Barron, Yu Xin}
\date{}

\begin{document}
\maketitle
\begin{abstract}
Let $\pi$ be a simple supercuspidal representation of the quasi-split unramified even unitary group with respect to an unramified quadratic extension $E/F$ of $p$-adic fields. We compute the Rankin-Selberg gamma factor for rank-$1$ twists of $\pi$ by a tamely ramified character of $E^\times$. For non-dyadic cases, the gamma factor can also be recovered by considering the endoscopic lift of the representation. For the dyadic case, the result is original and we expect to extend the results on the Langlands parameter of the simple supercuspidal representations to the dyadic case with our computation.
\end{abstract}

\tableofcontents


\section{Introduction}
Let $F$ be a local non-archimedean field of characteristic $0$ and of residue characteristic $p>0$. Let $l$ be a positive integer and $\u_{2l}$ be the quasi-split unitary group of rank $l$ with respect to an unramified quadratic extension $E/F$.
The simple supercuspidal representations of reductive groups are the class of supercuspidal representations of minimal positive depth. Such representations for arbitrary reductive groups were developed and studied in \cite{GR10} and generalized to a class of supercuspidal representations called the epipelagic representations in \cite{RY14}. Simple supercupsidal representations are important towards obtaining results on arbitrary supercuspidal representations. In this work we compute the Rankin-Selberg gamma factor for simple supercuspidal representations of the unitary group $\u_{2l}(F)$.

The gamma factor $\gamma(s,\pi\times \tau, \Upsilon, \psi)$ was defined by Morimoto \cite{Mor23} by extending and elaborating on the Rankin-Selberg integrals of Ben-Artzi and Soudry \cite{BS,BS2}, where $\psi$ is a nontrivial character of $E$ which is invariant under the action of the Galois group $\mathrm{Gal}(E/F)$ of the field extension $E/F$, $\pi$ is a generic representation of $\u_{2l}(F)$, $\tau$ is a tamely ramified character of $E^\times$ and $\Upsilon$ is a character of $E^\times$ extending the quadratic character $\omega_{E/F}$ associated with the field extension $E/F$ in the class field theory.

We fix a uniformizer $\varpi \in F$ and a level one additive character $\psi_F$ of $F$. The simple supercuspidal representation $\pi$ is parameterized by two parameters: a character $\omega^1$ of $k_E^1$ and an element $b \in \calO_F^\times$ modulo $1+\frakp_F$, where $k_E^1$ is the multiplicative group of norm one elements in the residue field of $E$. See section \ref{SSC} for more details. Our main theorem is the following.
\begin{thm}[Theorem \ref{main1}]
Let $\pi = {\pi_{\omega^1,b}}$ be a simple supercuspidal representation and let $\tau$ be a tamely ramified characters of $E^\times$. Let $\Upsilon$ be a tamely ramified character extending $\omega_{E/F}$. Then 
\[
\gamma(s, \pi \times \tau, \Upsilon, \psi) = \omega^1(-1) \tau(-1)^{l} q_F^{-2s+1} \tau(b^{-1}\varpi).
\]
Equivalently,
\[
\gamma^{\text{Sh}}(s, \pi \times \tau, \psi) = -\omega^1(-1) \tau(-1)^{l} q_F^{-2s+1} \tau(b^{-1}\varpi).
\]
Here $\gamma^{\Sh}$ is the gamma factor defined by Shahidi \cite{Sh3}.
\end{thm}

In particular, for non-dyadic case the rank-$1$ twisted gamma factor of a simple supercuspidal representation can be recovered by considering its endoscopic lift (cf. \cite{AL14,Mok15,Oi18,KK05}). For the dyadic case, the construction of the endoscopic lift in \cite{Oi18} is invalid and our result is original.
Hence one possible application of this result is to the explicit description of the Landlands parameter $\varphi_\pi$ of $\pi$. Recall that by the work of Asgari, Cogdell and Shahidi \cite{ACS} Theorem 3.2, there is a unique irreducible admissible generic representation $\Pi$ of $\mathrm{GL}_{2l}(E)$ which is a functorial lift of $\pi$, and in particular $\gamma^{\Sh}(s,\pi \times \tau, \psi) = \gamma(s,\Pi \times \tau,\psi)$. 
In particular, our theorem proves that there are no tamely ramified characters in the cuspidal support of $\Pi$, which is in line with the expectation that $\Pi$ is simple supercuspidal.

Similar results were obtained for the groups $\SO_{2n+1}$ by Adrian \cite{Ad15} and for $\Sp_{2n}$ and $\SO_{2n+1}$ by Adrian and Kaplan \cite{AK19,AK21}. The present case is similar to the odd orthogonal case in that the gamma factor has no poles, as opposed to the cases of the other groups. See also the recent work of Adrian, Henniart, Kaplan and Oi \cite{AHKO} where the Langlands parameter of a simple supercuspidal representation of $\SO_{2n}$ in the dyadic case was found using, among other important results, a similar computation of the gamma factor.

\subsection*{Acknowledgement}
The authors express their deep gratitude to Eyal Kaplan for suggesting this project and for his dedicating help. The authors also want to express their gratitude to Moshe Adrian, Baiying Liu, Kazuki Morimoto and Masao Oi for answering related questions.

Xin was supported by the ISRAEL SCIENCE FOUNDATION (grant No. 376/21) as the fellow’s post-doctoral fellowship at Bar-Ilan University in the department of mathematics. Barron was supported by a Zuckerman Post-doctoral fellowship at Bar-Ilan University and by the ISRAEL SCIENCE FOUNDATION (grant No. 376/21).

\section{Notations and Preliminaries}
\subsection{Notations}
Let $F$ be a $p$-adic field. We denote by $\mathcal{O}_F$ and $\mathfrak{p}_F$ for the ring of integers of $F$ and the maximal ideal, respectively. We fix a uniformizer $\varpi$, and we normalize the absolute values $|\cdot|_F$ on the local fields as
\(
|\varpi|_F = q_F
\)
where $q_F$ is the size of the residue field $k_F$ of $F$, so that for any Haar measure $dx$ on $F$ we have
\(
d(ax) = |a|_F dx.
\)
Let $E$ be an unramified quadratic extension of $F$. We apply the above notations to $E$ as well and note that 
\begin{itemize}
 \item $\varpi$ is also a uniformizer for $E$.
 \item $q_E = q_F^2$.
\end{itemize}

Let $l$ be a positive integer and let $\U_{2l}$ be the unitary group on $2l$ variables, defined by
\[
\U_{2l}(F) = \U_{E/F,2l}(F)=\{g \in \GL_{2l}(E) : {^t\bar{g}} J_{2l} g  = J_{2l}\},
\]
where
\[
J_{2l} = \begin{pmatrix} & w_l\\ -w_l & \end{pmatrix}
\text { and }
w_l = \begin{pmatrix}
    &&1\\
    &\iddots&\\
    1&&
\end{pmatrix}.
\]
Let $\u_{2l} = \U_{2l}(F)$. Fix the Borel subgroup $B_l = T_l \ltimes N_l$ of upper triangular matrix in $\u_{2l}$, where $T_l$ is the maximal torus and $N_l$ is the unipotent radical. 
For $1 \leq k \leq l$, let $Q_k < \u_{2l}$ be the standard maximal parabolic subgroup whose Levi part is isomorphic to $\GL_k(E) \times \u_{2(l-k)}$, and $\delta_{Q_k}$ be the modular character of $Q_k$. In particular $Q_l$ is the Siegel parabolic subgroups of $\u_{2l}$.
We denote $a^* = w_l {^t\bar{a}^{-1}} w_l$ and $^x y = x^{-1}y x$ for $a \in \GL_k(E)$ and $x,y \in \u_{2k}$.

\subsection{The Rankin-Selberg gamma factor}
In this section, we will introduce necessary notations for the rank-$1$ twisted Rankin-Selberg gamma factor for the unramified quasi-split unitary group. For the general case with a rank-$k$ twist, see \cite{Mor23}. 

We fix an additive character $\psi_F$ of $F$, and let $\psi = \psi_F \circ \mathrm{tr}_{E/F}$, where $\mathrm{tr}_{E/F}$ is the trace map of the field extension $E/F$. 
In general, a generic character on $N_l$ is of the form $$\psi_{{N_l},\alpha} = \psi_E\left(\sum_{i=1}^{l-1} n_{i,i+1} - \frac{\alpha}{2} n_{l,l+1}\right)$$
where $\alpha \in F^\times$. The theory of the Rankin-Selberg gamma factor of the unitary group were developed in \cite{Mor23}. Our work will be based on the theory for the generic character $\psi_{N_l} := \psi_{N_l,1}$, which was developed in \textit{loc. cit.}.

Let $\Upsilon$ be a character of $E^\times$ satisfying $\Upsilon|_{F^\times} = \omega_{E/F}$, where $\omega_{E/F}$ is 
the quadratic character associated with the quadratic extension $E/F$ in the class field theory.
The Weil representation $\omega_{\psi,\Upsilon}$ is a group action of $\u_{2k} \ltimes H_k$ on the space $\mathcal{S}(E^k)$
of Bruhat-Schwartz functions, where $H_k$ is the so called Heisenberg group. In particular, the action of $\u_{2k}$ when $k=1$ is given by
\begin{align}
    \omega_{\psi,\Upsilon}\left(
    \begin{pmatrix}
        a & \\ & a^*
    \end{pmatrix}\right)\phi(\xi) &= \Upsilon(a)|a|_E^{\frac{1}{2}}\phi(\xi a),\label{weil:diag}\\
    \omega_{\psi,\Upsilon}\left(
    \begin{pmatrix}
        1 & x \\ & 1
    \end{pmatrix}\right)\phi(\xi) &= \psi(\frac{1}{2}(\xi \bar{\xi x})\phi(\xi),\label{weil:uni}\\
    \omega_{\psi,\Upsilon}\left(\pmat{}{1}{-1}{}\right)\phi(\xi) = \hat{\phi}(\xi)&:=\int_{E}\phi(x)\psi(\bar{x}\xi) dx \label{weil:weyl},
\end{align}
where $a^* = {^t\sigma(a)}$ and
where $dx$, in defining the Fourier transform, is the self-dual Haar measure with respect to $\psi$. 

Let $\pi$ be a generic representation of $\u_{2l}$ and $\tau$ a character of $E^\times$. Let $\tau^*$ be the character defined by $\tau^*(a) = \tau(a^*)$. We realize $\pi$ in its Whittaker model as $\mathcal{W}(\pi,\psi_{N_l}^{-1})$. For every $s \in \C$, we consider the parabolic induction of $\tau$ to $\u_2$ and realize it as the space $V(\tau,s)$ of smooth, complex-valued functions $f_s(-,-)$ on $\u_{2}\times \GL_1(E)$ such that
\begin{align*}
f_s\left(\pmat{m}{}{}{m^*}\pmat{1}{u}{}{1}g,a\right) &= \delta_{Q_1}\left(\pmat{m}{}{}{m^*}\right)|m|_E^{s-1/2}f_s(g,am) \\&= |m|_E^{s}f_s(g,am),
\end{align*}
and $f_s(g,a) = \tau(a)$ for all $g \in \u_2$ and $a \in E^\times$.

Let 
\(
W \in \mathcal{W}(\pi,\psi_{N_l}^{-1}),~
f_s \in V(\tau,s)\text{ and }
\phi \in \mathcal{S}(E).
\)
The Rankin-Selberg integral for $\pi \times \tau$ is defined by
\[ \mathcal{L}(W,f_s,\phi) =
\int_{N_1 \backslash \u_2} \int_{R^{l,1}}\int_{X_1} W( ^{w_{l-1,1}} (rxg))f_s(g) [\omega_{\psi^{-1},\Upsilon^{-1}}(g) \phi](x) dx dr dg
\]
when $l > 1$. Here
\begin{align*}
R^{l,1} &= \left\{\begin{pmatrix}
I_{l-2} & 0 & r & & &\\
 & 1 & & & &\\
 &&1&&&\\
 &&&1&&r^\prime\\
 &&&&1&0\\
 &&&&&I_{l-2}
\end{pmatrix} \in \U_{2l}
\right\},\\
X_1 &=\left\{ \begin{pmatrix}
1 & x & 0 & 0\\
0 & 1 & 0 & 0\\
0 & 0 & 1 & x^\prime\\
0 & 0 & 0 & 1
 \end{pmatrix}\right\},\\
w_{l-1,1} &= \begin{pmatrix}
 & I_{l-1}& &\\
 1& & &\\
 & & &1\\
 & & I_{l-1} &\\
\end{pmatrix},
\end{align*}
where $r^\prime ={^t\sigma(r)}$ and $x^\prime = \sigma(x)$.
The integral is absolutely convergent for $\mathrm{Re} s \gg 0$ and admits a meromorphic continuation to $\C(q^{-s})$. 

Let $M(\tau,s):V(\tau,s) \rightarrow  V(\tau^*,1-s)$ be the intertwining operator, given by the meromorphic continuation of the following integral:
\[
[M(\tau,s)f_s] (h,a) = \int_F f_s\left(\pmat{}{-1}{1}{} \pmat{1}{u}{}{1}h,-a^*\right) du,
\]
where $h \in \u_{2}$, $a \in E^\times$ and $du$ is the Haar measure with respect to $\psi_F$.

Let $\Lambda(\tau,s)$ be the following Whittaker functional defined by the analytic continuation of
\[
\Lambda(\tau,s)f = \int_{F} f\left(\pmat{}{-1}{1}{} \pmat{1}{u}{}{1}, 1\right) \psi_F^{-1}(u) du
\]
and more directly by
\begin{equation}\label{whi.fun.AC}
\lim_{N \rightarrow \infty} \int_{\frakp_F^{-N}} f\left(\pmat{}{-1}{1}{} \pmat{1}{u}{}{1}, 1\right) \psi_F(u) du
\end{equation}
where $f$ is a flat section.
By the uniqueness of the Whittaker functional, there is a constant $C(s,\tau,\psi)$ such that
\begin{equation}\label{localcoefficient}
\Lambda(\tau,s) = C(s,\tau,\psi) \cdot \Lambda(\tau^*,1-s) \circ M(\tau,s).
\end{equation}
We define the normalized intertwining operator as $M^*(\tau,s) = C(s,\tau,\psi)M(\tau,s)$. It admits the functional equation
\begin{equation}
\int_{F} f_s\left(\pmat{}{-1}{1}{} \pmat{1}{u}{}{1}, 1\right) \psi_F^{-1}(u) du = \int_{F} [M^*(\tau,s)f_s]\left(\pmat{}{-1}{1}{} \pmat{1}{u}{}{1}, 1\right) \psi_F^{-1}(u) du.
\end{equation}
Let $\mathcal{L}^*(W,f_s,\phi) = \mathcal{L}(W,M^*(\tau,s)f_s,\phi).$ The Rankin-Selberg integrals ${\mathcal{L}^*(W,f_s,\phi)}$ and ${\mathcal{L}(W,f_s,\phi)}$ are trilinear maps satisfying a certain collection of functional equations. Such trilinear maps form a one-dimensional vector space (cf. \cite{Mor23}). Hence they are proportional and we denote
\[
\mathcal{L}^*(W,f_s,\phi) = \gamma_0(s,\tau,\Upsilon,\psi)\mathcal{L}(W,f_s,\phi).
\]
We further denote the normalized gamma factor as
\begin{equation}\label{defn:norgam}
\gamma(s,\tau,\Upsilon,\psi) = \omega_\pi(-1)\tau(-1)^l\gamma_0(s,\tau,\Upsilon,\psi)
\end{equation}
where $\omega_\pi$ is the central character of the representation $\pi$.

It worth mentioning that the Rankin-Selberg gamma factor $\gamma(s,\pi \times \tau, \Upsilon, \psi)$ in this paper is the normlized one $\Gamma(s,\pi \times \tau, \Upsilon, \psi)$ in \cite{Mor23}.

In the rest of the section, we compute $C(s,\tau,\psi)$.
\begin{prop}
$C(s,\tau,\psi) = \tau(-1)\gamma(2s-1,\tau|_{F^\times},\psi_F).$
\end{prop}

\begin{proof}
For convenience, we rewrite the functional equation \eqref{localcoefficient} as
\begin{equation}\label{func.eq.C}
\Lambda^\prime \circ M = C(s,\tau,\psi)^{-1} \Lambda.
\end{equation}
Applying \eqref{whi.fun.AC}, we have
\[
\Lambda^\prime \circ M(f_s) = \lim_{N\rightarrow \infty} \int_{\frakp_F^{-N}} \int_F f_s\left(w_0 \pmat{1}{y}{}{1} w_0  \pmat{1}{x}{}{1},-1 \right) \psi_F^{-1}(x) dy dx.
\]
By the compactness of $\frakp_F^{-N}$ and the absolute convergence of the integral when $\mathrm{Re}~s \gg 0$, we can change the order of integrals and obtain
\[
\Lambda^\prime \circ M(f_s) = \lim_{N\rightarrow \infty} \int_F \int_{\frakp_F^{-N}} f_s\left(w_0 \pmat{1}{y}{}{1} w_0  \pmat{1}{x}{}{1},-1 \right)  \psi_F^{-1}(x) dx dy.
\]
Observing that 
\(
w_0 \pmat{1}{y}{}{1} w_0  \pmat{1}{x}{}{1} = \pmat{y^{-1}}{-1}{}{y} w_0 \pmat{1}{x-y^{-1}}{}{1}
\) and applying the definition of $f_s$, we have 
\[
\lim_{N\rightarrow \infty} \int_F \int_{\frakp_F^{-N}} \tau(-y^{-1})|y|_E^{-s}f_s\left(w_0 \pmat{1}{x-y^{-1}}{}{1},1 \right)  \psi_F^{-1}(x) dx dy.
\]
Applying the change of variable $y \mapsto y^{-1}$ and noticing that $\frac{dy}{|y|_F}$ is invariant under the change of variable, we have
\begin{align}
\lim_{N\rightarrow \infty} \int_F \int_{\frakp_F^{-N}} \tau(-y)|y|_E^{s-1} f_s\left(w_0 \pmat{1}{x-y}{}{1},1 \right)  \psi_F^{-1}(x) dx {dy}. \label{ddd}
\end{align}
Applying \eqref{func.eq.C} and \eqref{ddd}, we obtain
\[
\lim_{N\rightarrow \infty} \int_F \int_{\frakp_F^{-N}} \tau(-y)|y|_E^{s-1} \phi(x-y)  \psi_F^{-1}(x) dx {dy} = C(s,\tau,\psi) ^{-1} \lim_{N \rightarrow \infty} \int_{\frakp_F^{-N}} \phi(x) \psi_F^{-1}(x) dx
\]
where we denote $\phi(x) = f_s\left(w_0 \pmat{1}{x}{}{1}\right)$.
Further, by reorganizing the left hand side, we obtain
\begin{equation}\label{loc.coeff.fun.eq}
\lim\limits_{N\rightarrow \infty} \int_F \tau(-y) |y|_E^{s-1} \psi_F^{-1}(y) \int_{\frakp_F^{-N}} \phi(x-y) \psi_F^{-1}(x-y) dx dy \atop = C(s,\tau,\psi)^{-1} \lim\limits_{N \rightarrow \infty} \int_{\frakp_F^{-N}} \phi(x) \psi_F^{-1}(x) dx.
\end{equation}
Let $f_s$ be such that $\phi$ is the characteristic function on $\frakp_F^M$ where $N \geq -M$. Such function $f_s$ exists, and we note that
\begin{equation}\label{id1}
\int_{\frakp_F^{-N}} \phi(x-y) \psi_F^{-1}(x-y) dx = \vol(\frakp_F^M) \mathbf{1}_{\frakp_F^{-N}}(y).
\end{equation}
Now apply \eqref{id1} to \eqref{loc.coeff.fun.eq}. Then we have
\begin{align*}
C(s,\tau,\psi)^{-1} &= \lim_{N\rightarrow \infty} \int_{\frakp_F^{-N}} \tau(-y) |y|_E^{s-1} \psi_F(-y) dy\\
&= \int_{F^\times} \tau(y) \psi_F(y) |y|_F^{2s-2} dy\\
&= \gamma(2-2s,\tau|^{-1}_{F^\times},\psi_F).
\end{align*}
We refer to \cite{Bum97}, Exercise 3.1.9, for the last identity. Lastly, we apply the functional equation for Tate's gamma factors (cf. \cite{BH06} (23.4)).
\end{proof}


\noindent\textbf{Remark:} For brevity, we skipped the discussion when $l=1$ in this section and will do the same in the later sections. However, our main result Theorem \ref{main1} will hold for $l=1$ as well and the proof will remain similar and simpler.

\subsection{Simple supercuspidal representations of $\u_{2l}$}\label{SSC}
In this section, we fix a level one additive character $\psi_F$ and let $\psi = \psi_F \circ \mathrm{tr}_{E/F}$. 
\begin{lem}
The additive character $\psi$ on $E$ is of level one. 
\end{lem}
\begin{proof}
Let $x \in \frakp_E$, then $\mathrm{tr}(x) \in \frakp_F$. Hence $\psi(\mathrm{tr}(x)) = 1$ and the level of $\psi$ is not greater than one.

Let $x \in \calO_F$ be such that $\psi(x) \neq 1$.
When $p \neq 2$, $\psi(x) \neq 1$ shows that the level of $\psi$ is exactly one.
When $p = 2$, there exists an element $a \in \calO_E^\times$ such that $\mathrm{tr}_{E/F}(a) = -1$ (cf. \cite{Tits79}, p. 41).
Hence for any $x \in \calO_F^\times$, there exists $a\in \calO_E^\times$ such that $\mathrm{tr}_{E/F}(a) = x$. This shows the existence of an element $x$ where $\psi(x) \neq 1$ and the level is exactly one.
\end{proof}

Let $I=I_{\u_{2l}}$ be the standard Iwahori subgroup of the unitary group $\u_{2l}$ and let $I^{+}$ and $I^{++}$ be its filtration groups. 
For any $b \in \calO_F^\times$ modulo $1+\frakp_F$, we define a character on $I^+/I^{++} \cong k_F^{\oplus (l-1)} \oplus k_E \oplus k_E$ by
\begin{equation}
\chi_b(y) = \psi(y_{1,2} + \dots + y_{l-1,l} + \frac{y_{l,l+1}}{2} + \frac{by_{2l,1}\varpi^{-1}}{2} )\label{chib},
\end{equation}
as that in \cite{Oi18}.
In parameterizing the simple supercuspidals of $\u_{2l}$, we define a character $\chi_{\omega^1,b}$ by
\begin{align*}
\chi_{\omega^1,b}(zy) = \omega^1([z])\chi_b(y)
\end{align*} 
where $z \in Z_{\U_{2l}}$, $[z]$ its projection onto $k^1_E$, $y \in I^+$, and $\omega^1$ a character on $k_E^1 := \mathrm{ker}(N_{k_E/k_F})$. Here $N_{k_E/k_F}$ is the norm map with respect to the field extension $k_E/k_F$.

We define the representation $\pi_{\omega^1,b} = \cind_{ZI^+}^{\u_{2l}} \chi_{\omega^1,b}$ by compact induction where $\omega^1 \in (k_E^1)^\vee$, $b \in k_F^\times$. Varying $\omega^1$ and $b$ yields the simple supercuspidal irreducible representations of $\u_{2l}$, which we have $q^2-1$ in total. Since $b$ and $\varpi$ appear as $b\varpi^{-1}$ in the parametrization, one can also parameterize the representations by fixing $b = 1$ and vary the uniformizer $\varpi$ modulo $1+\frakp_F$.

In order to have the Rankin-Selberg integrals, we need the genericity of the representations.
\begin{lem}
 The simple supercuspidal representation $\pi_{\omega^1,b}$ is generic with respect to $\psi_{N_l}^{-1}$. 
 \end{lem}
\begin{proof}
It is easy to see that any representation $\pi_{\omega^1,b}$ is $\psi_{N_l,-1}$-generic, with the following function a Whittaker function:
\[
W_0(g) = \begin{cases}
\psi_{N_l,-1}(u) \omega^1([z]) \chi_b(y) & \text{if } g = uzy \in N_l Z I^+,\\
0 & \text{otherwise}.
\end{cases}
\]
We note $\psi_{N_l,-1}$ is rationally conjugate to $\psi_{N_l}^{-1}$. Indeed, for example, by putting $$\iota:= \mathrm{diag}((-1)^{-l+1} ,(-1)^{-l+2} ,\ldots,1,1,\ldots,(-1)^{l-2} ,(-1)^{l-1}) \in T_l$$
we get
\begin{align*}
\psi_{N_l,-1}^\iota(u) &= \psi_{N_l, -1}(\iota u \iota^{-1})\\
&= \psi(-u_{12} - u_{23} - \ldots - u_{l-1,l} + \frac{1}{2} u_{l,l+1})\\
&= \psi_{N_l}^{-1}(u)
\end{align*} 
for any $u \in N_l$.
\end{proof}

In general, if $\iota$ is an element in $\u_{2l}$ and $\pi^\iota$ is a representation of $\u_{2l}$, we define $\pi^\iota$ as the representation with the action $\pi^\iota (g) = \pi(\iota g \iota^{-1})$ on the same space of $\pi$. Such representation $\pi^\iota$ is isomorphic to $\pi$ and
\[
\gamma(s, \pi \times \tau, \psi) = \gamma(s, \pi^\iota \times \tau, \psi)
\]
whenever $\pi$ is generic. In particular, when $\pi = \pi_{\omega^1,b}$ and $\iota$ is the one in the proof above, we will compute the right hand side of the above identity. The image of $W_0$ under the isomorphism $\mathcal{W}(\pi, \psi_{N_l, -1}) \rightarrow \mathcal{W}(\pi^\iota, \psi_{N_l}^{-1})$ defined by $W \mapsto W^\iota$, where $W^\iota(g) = W(\iota g \iota^{-1})$, gives us a Whittaker function in the latter.

\section{Computation of the gamma factors}
\subsection{Some computations regarding $B^-$}
Given a parabolic subgroup $P$ of $\u_{2l}$, we denote by $P^-$ its opposite parabolic subgroup. In this section, we summarize some computations which are needed in the later sections.
We write elements in the Borel subgroup $B_1^-$ of $\u_2$ as
\begin{equation}\label{b=ac}
b = \begin{pmatrix}
    a &  \\ \bar{a}^{-1}c & \bar{a}^{-1}
\end{pmatrix} =\begin{pmatrix}
    a &  \\  & \bar{a}^{-1}
\end{pmatrix}\begin{pmatrix}
    1 &  \\ c & 1
\end{pmatrix} 
\end{equation}
where $a \in E^\times$ and $c \in F$.

\begin{lem}\label{weilb}
    Let $\phi \in \mathcal{S}(E)$ and $b$ be given as \eqref{b=ac}. Then
\begin{equation*}
    [\omega_{\psi^{-1},\Upsilon^{-1}}(b)\phi](x) = \Upsilon^{-1}(-a)|a|_E^{\frac{1}{2}}\int_E \psi_F \left(-\frac{1}{2}c y \bar{y}\right)\psi(ax\bar{y})\int_E \phi(z)\psi(-y\bar{z}) dz dy.
\end{equation*}
Here $dz$ and $dy$ are self-dual with respect to $\psi$.
\end{lem}
\begin{proof} Applying (\ref{weil:diag})-(\ref{weil:weyl}), we see that
\begin{align*}
    [\omega_{\psi^{-1},\Upsilon^{-1}}(b)\phi](x)
    &=\omega_{\psi^{-1},\Upsilon^{-1}}\left(
    \begin{pmatrix}
        -a & \\  & -\bar{a}^{-1}
    \end{pmatrix}
    \begin{pmatrix}
         & 1\\ -1 & 
    \end{pmatrix}
    \begin{pmatrix}
        1 & -c\\  & 1
    \end{pmatrix}
    \begin{pmatrix}
         & 1\\ -1 & 
    \end{pmatrix}\right)\phi(x)\\
    &=\Upsilon^{-1}(-a)|a|_E^{\frac{1}{2}}\omega_{\psi^{-1},\Upsilon^{-1}}\left(
    \begin{pmatrix}
         & 1\\ -1 & 
    \end{pmatrix}
    \begin{pmatrix}
        1 & -c\\  & 1
    \end{pmatrix}
    \begin{pmatrix}
         & 1\\ -1 & 
    \end{pmatrix}\right)\phi(-ax)\\
    &=\Upsilon^{-1}(-a)|a|_E^{\frac{1}{2}}\int_E\left(\omega_{\psi^{-1},\Upsilon^{-1}}\left(
    \begin{pmatrix}
        1 & -c\\  & 1
    \end{pmatrix}
    \begin{pmatrix}
         & 1\\ -1 & 
    \end{pmatrix}\right)\phi\right)(y)\psi(ax\bar{y})dy \\
     &=\Upsilon^{-1}(-a)|a|_E^{\frac{1}{2}}\int_E \psi_F \left(-\frac{1}{2}c y \bar{y}\right)\left(\omega_{\psi^{-1},\Upsilon^{-1}}\left(
    \begin{pmatrix}
         & 1\\ -1 & 
    \end{pmatrix}\right)\phi\right)(y)\psi(ax\bar{y})dy \\
     &=\Upsilon^{-1}(-a)|a|_E^{\frac{1}{2}}\int_E \psi_F \left(-\frac{1}{2}c y \bar{y}\right)\psi(ax\bar{y})\int_E \phi(z)\psi(-y\bar{z}) dz dy.
\end{align*}
\end{proof}

Given a Haar measure $dx$ (resp. $d^\times x$) on $F$ (resp. $F^\times$), we denote the volume of a measurable set $S$ in $F$ (resp. $F^\times$) by $\vol_F(S) =\int_S dx$ (resp. $\volx_F(S) = \int_S d^\times x$). Similarly we define $\vol_E$ and $\volx_E$.

\begin{lem}\label{lem:Mtausfsb1}
Let $f_s \in V(\tau,s)$ and $b$ be given as (\ref{b=ac}), with $a \in 1+\frakp_E$ and $c \neq 0$. Then
\[
[M(\tau,s)f_s](b,1) = \vol_F(\calO_F^\times) \int_{F^\times}|u|_E^{-s+\frac{1}{2}} \tau(u^{-1}) f_s\left(
\begin{pmatrix} 1 & \\u^{-1}\bar{a}a + c& 1\end{pmatrix},-1\right)d^\times u.
\]
Here $d^\times u$ is the Haar measure of $F^\times$ normalized as $\volx_F(\calO_F^\times) = 1$.
\end{lem}
\begin{proof}
By definition of the standard intertwining operator, for any $g \in \u_2$,
\[
[M(\tau,s)f_s](g,1) = \int_{F} f_s(w_1^{-1}ug,-1) du,
\]
where $du$ is the self-dual Haar measure with respect to $\psi_F$. In particular,
\begin{align*}
[M(\tau,s)f_s](b,1) &= \int_F f_s\left(
\begin{pmatrix} & -1\\1& \end{pmatrix}
\begin{pmatrix} 1& u\\&1 \end{pmatrix}
\begin{pmatrix} a & \\\bar{a}^{-1}c& \bar{a}^{-1}\end{pmatrix},-1\right)du\\
&=\int_F f_s\left(
\begin{pmatrix} 1& u^{-1}\\& 1\end{pmatrix}
\begin{pmatrix} & -1\\1& \end{pmatrix}
\begin{pmatrix} 1& u\\& 1\end{pmatrix}
\begin{pmatrix} a & \\\bar{a}^{-1}c& \bar{a}^{-1}\end{pmatrix},-1\right)du\\
&=\int_F f_s\left(
\begin{pmatrix} u^{-1}a & \\a+uc\bar{a}^{-1}& u\bar{a}^{-1}\end{pmatrix},-1\right)du.
\end{align*}
We note that $u = \bar{u}$. Hence
\begin{align*}
[M(\tau,s)f_s](b,1)&=\int_F f_s\left(
\begin{pmatrix} u^{-1}a & \\a+uc\bar{a}^{-1}& \bar{ua}^{-1}\end{pmatrix},-1\right)du\\
&=\int_F f_s\left(
\begin{pmatrix} u^{-1}a & \\& (\bar{u^{-1}a})^{-1}\end{pmatrix}
\begin{pmatrix} 1 & \\u^{-1}\bar{a}a + c & 1\end{pmatrix},-1\right)du\\
&=\int_F |u^{-1}|_E^{s} f_s\left(
\begin{pmatrix} 1 & \\u^{-1}\bar{a}a + c & 1\end{pmatrix},-u^{-1}a\right)du\\
&=\int_F |u^{-1}|_E^{s} \tau(u^{-1}a) f_s\left(
\begin{pmatrix} 1 & \\u^{-1}\bar{a}a + c& 1\end{pmatrix},-1\right)du.
\end{align*}
Since $\tau$ is tamely ramified, $\tau(a) = 1$. Hence
\begin{align*}
[M(\tau,s)f_s](b,1) &= \int_F |u|_E^{-s} \tau(u^{-1}) f_s\left(
\begin{pmatrix} 1 & \\u^{-1}\bar{a}a + c& 1\end{pmatrix},-1\right)du\\
&=\vol_F(\calO_F^\times) \int_{F^\times}|u|_E^{-s+\frac{1}{2}} \tau(u^{-1}) f_s\left(
\begin{pmatrix} 1 & \\u^{-1}\bar{a}a + c& 1\end{pmatrix},-1\right)d^\times u.
\end{align*}
Here we recall that $d^\times u = \vol_F(\calO_F^\times)^{-1} \frac{d u}{|u|_F}$.
\end{proof}

\subsection{Computation of the gamma factor}\label{comp}
Let $\tau$ a tamely ramified unitary quasi-character of $E^\times$ and $\pi = \pi_{\omega^1,1}$ a simple supercuspidal representation. We assume $\Upsilon$ is tamely ramified.
Let
\[
W_0(g) = 
\begin{cases}
    \psi_{N_l,-1}(u)\omega^1([z])\chi_1(y) & g = uzy, u \in N_l, z\in Z, y \in I^+\\
    0 & \text{otherwise},
\end{cases}
\]
in $\mathcal{W}(\pi,\psi_{N_l, -1})$ and take $W = W_0^\iota$ in $\mathcal{W}(\pi,\psi^{-1}_{N_l})$. Let
\[
f_s(g,a)=
\begin{cases}
    |m|_E^s\tau(am), & g = \begin{pmatrix} m & \\ &m^* \end{pmatrix} uy \in TN_1 I^{++}, a \in E^\times\\
    0, & \text{otherwise}.
\end{cases}
\]
in $V(\tau,s)$.
Let $\phi \in \mathcal{S}(E)$ be the characteristic function of $\mathfrak{p}_E$.

\begin{lem}\label{entriesinp}
Let $b$ be as that in (\ref{b=ac}) and assume $^{\iota w_{{l-1},1}}(rxb) \in N_lZI^+$. We have $a \in 1+ \mathfrak{p}_E$, $x \in \mathfrak{p}_E$, $c \in \mathfrak{p}_F$ and the coordinates of $r$ belong in $\mathfrak{p}_E$. Consequently, we have $^{\iota w_{{l-1},1}}(rxb) \in I^+$.
\end{lem}

\begin{proof}
By straight forward computation, $^{\iota w_{{l-1},1}}(rxb)
=$

\begin{equation}\label{formula:w.rxb}
\begin{pmatrix}
    a &&&&&&&&\\
    (-1)ar_1 & 1 &&&&&&&\\
    \vdots &&\ddots&&&&&&\\
    (-1)^{l-2} ar_{l-2} &&&1&&&&&\\
    (-1)^{l-1}ax &&&&1&&&&\\
    0 &&&&&1&&&\\
    \vdots &&&&&&\ddots&&\\
    0 &&&&&&&1&\\
    \bar{a}^{-1}c &0&\dots & 0 &(-1)^{l-1} \sigma(x) & (-1)^{l-2} a \sigma(r_{l-2}) &  \dots & (-1) a\sigma(r_{1}) & \bar{a}^{-1}
\end{pmatrix}
\end{equation}
for all $l \geq 2$. 

We first note that $N_lZI^+$ is a disjoint union
$$N_lZI^+ = \coprod_{z \in \mathcal{O}_E^\times/1+\mathfrak{p}_E} N_lzI^+.$$
Indeed, it is obvious $N_lZI^+$ is the union. To see the sets $N_lzI^+$ and $N_lz^\prime I^+$ are disjoint, we simply take two arbitrary elements from each and compare their $(2l,2l)$-entries.
We then observe that if $^{w_{l-1,1}}(rxb)$ is in $N_lZI^+$, then it must lie in $N_lI^+$. Indeed, we can let $^{w_{l-1,1}}(rxb) = uzy$ and compare the $(2,2)$-entries of $u^{-1}(^{w_{l-1,1}}(rxb))$ and $zy$. The comparison of the last rows of $^{w_{l-1,1}}(rxb)$ and $uy$ completes the proof of the first statement. 
We note that the second statement follows from the first one, by the explicit description of $I^+$.
\end{proof}


\begin{lem}\label{lem:r-sint}
Let $W,f_s$ and $\phi$ be as above. The Rankin-Selberg integral
\[
\mathcal{L}(W,f_s,\phi) = 
\vol_E(\calO_E) \vol_E(\frakp_E)^l\vol(\frakp_F^2) \volx_E(1+\frakp_E).
\]
\end{lem}

\begin{proof}
Rewrite the outer integral over $B^-$ as
\[
\int_{B^-}\int_{E^{\oplus l-2}}\int_{E} W(^{w_{l-1,1}}(rxb))[\omega_{\psi^{-1},\Upsilon^{-1}}(b)\phi](x) f_s(b,1) \delta_B(b) dx dr db.
\]
By the definitions of $W$ and $W_0$, \eqref{chib} and \eqref{formula:w.rxb},
\[
W(^{w_{l-1,1}}(rxb)) =
\begin{cases}
\psi_F(c\varpi^{-1}) & \text{ if } x \in \frakp_E, r \in \frakp_E^{\oplus(l-2)}, a \in 1+ \frakp_E, c \in \frakp_F\\
0 & \text{ otherwise}
\end{cases}.
\]
Hence the integral is
\[
\int_{S}\int_{\mathfrak{p}_E^{\oplus l-2}}\int_{\mathfrak{p}_E} \psi_F(c\varpi^{-1})[\omega_{\psi^{-1},\Upsilon^{-1}}(b)\phi](x) f_s(b,1) dx dr db
\]
where
$S=\left\{\begin{pmatrix} a & 0\\ \bar{a}^{-1}c & \bar{a}^{-1} \end{pmatrix} \mid a\in 1+\mathfrak{p}_E,c\in \mathfrak{p}_F, c \neq 0\right\}$. Here we note that $c=0$ gives a measure zero subset in $S$.
Apply Lemma \ref{weilb}:
\begin{align*}
[\omega_{\psi^{-1},\Upsilon^{-1}}(b)\phi](x) = \Upsilon^{-1}(-a)\int_E \psi_F\left(-c y \bar{y}\right)\psi(ax\bar{y})\int_{\mathfrak{p}_E}\psi(-y\bar{z}) dz dy.
\end{align*}
The inner integral vanishes for $y \notin \calO_E$, so the above equals
\begin{align*}
\Upsilon^{-1}(-a)\int_{\calO_E} \psi_F\left(-c y \bar{y}\right)\psi(ax\bar{y})\int_{\mathfrak{p}_E}\psi(-y\bar{z}) dz dy.
\end{align*}
The integrand is the characteristic function and the above integral equals $\vol_E(\calO_E)\vol_E(\frakp_E)$.
Hence
\begin{align}
\mathcal{L}(W,f_s,\phi) &= \int_{S}\int_{\mathfrak{p}_E^{\oplus l-2}}\int_{\mathfrak{p}_E} \psi_F(c\varpi^{-1})\Upsilon^{-1}(-a) \vol_E(\calO_E) \vol_E(\frakp_E) f_s(b,1) dx dr db \nonumber\\
&= \vol_E(\calO_E) \vol_E(\frakp_E)^l \int_S \psi_F(c\varpi^{-1}) \Upsilon^{-1}(-a) f_s(b,1) db.\label{comp:R-Sint}
\end{align}
Lastly, notice that when $a \in 1+\mathfrak{p}_E$ and $c \in \frakp_F$,
\[
f_s(b,1)=
\begin{cases}
    1, & c \in \mathfrak{p}_F^2,\\
    0, & \text{otherwise}.
\end{cases}
\]
\end{proof}

\begin{lem}\label{lem:int.act}

For $c \in \frakp_F^2$,
\[
[M(\tau,s)f_s] \left(\begin{pmatrix} 1 & \\ c & 1 \end{pmatrix},1\right) = 
\begin{cases}
q_F^{-2s+\frac{1}{2}} \tau(\varpi) (q_F-1)(L(2s-1,\tau|_{F^\times})-1) & \text{ if } \tau \text{ is unramified,}\\
0 & \text{ if } \tau \text{ is ramified. } 
\end{cases}
\]
For $c \in \varpi \calO_F^\times$,
\[
[M(\tau,s)f_s] \left(\begin{pmatrix} 1 & \\ c & 1 \end{pmatrix},1\right) = 
{q_F^{-2s+\frac{1}{2}}}\tau(c).
\]
\end{lem}

\begin{proof}
By Lemma \ref{lem:Mtausfsb1}, we have that
\begin{equation}\label{temp1}
\mathrm{LHS} = \vol(\calO_F^\times) \int_{F^\times}|u|_E^{-s+\frac{1}{2}} \tau(u^{-1}) f_s\left(
\begin{pmatrix} 1 & \\u^{-1} + c& 1\end{pmatrix},-1\right)d^\times u.
\end{equation}

When $c \in \frakp_F^2$, 
\[
f_s\left(\pmat{1}{}{u^{-1}+c}{1},-1\right)=f_s\left(\pmat{1}{}{u^{-1}}{1},-1\right)=
\begin{cases}
0, &u^{-1} \notin \frakp_F^2,\\
\tau(-1), &u^{-1} \in \frakp_F^2.
\end{cases}
\]
The integral in the above formula \eqref{temp1} is
\begin{align*}
\vol(\calO_F^\times)^{-1}\times\mathrm{LHS}
&=\int_{\{u|u^{-1} \in \frakp_F^2\}}|u^{-1}|_E^{s-\frac{1}{2}} \tau(-u^{-1}) d^\times u\\
&=\int_{\frakp_F^2}|u|_E^{s-\frac{1}{2}} \tau(-u) d^\times u\\
&=\sum_{n=2}^\infty \int_{\calO_F^\times} |\varpi^n u|_E^{s-\frac{1}{2}} \tau(-\varpi^n u) d^\times u\\
&=\sum_{n=2}^\infty (q_E^{-s+\frac{1}{2}} \tau(\varpi))^n \int_{\calO_F^\times}  \tau(-u) d^\times u.
\end{align*}
We notice that the last integral equals $0$ when $\tauF$ is ramified and $1$ when $\tauF$ is unramified.

When $c \in \varpi \calO_F^\times$, by our choice of the section $f_s$,
\[
f_s\left(\pmat{1}{}{u^{-1}+c}{1},-1\right)=
\begin{cases}
0, &\text{when }u^{-1}+c \notin \frakp_F^2,\\
\tau(-1), &\text{when }u^{-1}+c \in \frakp_F^2.
\end{cases}
\]
Hence \eqref{temp1} is
\begin{align*}
\vol(\calO_F^\times) \int_{\{u \in F^\times| u^{-1} + c \in \frakp_F^2\} }|u^{-1}|_E^{s-\frac{1}{2}}\tau(-u^{-1}) d^\times u.
\end{align*}
Apply the change of variable $u^{-1} = -u_0c$. We have
\[
\vol(\calO_F^\times) q_E^{-s+\frac{1}{2}}\tau(c)\int_{1+\frakp_F} \tau(u_0) d^\times u_0.
\]
Since $\tau$ is tamely ramified, we have
\begin{align*}
q_F^{-2s+\frac{1}{2}}\tau(c),
\end{align*}
where we applied $\vol(\calO_F^\times) = q_F^{\frac{1}{2}}-q_F^{-\frac{1}{2}}$ and $\vol^\times(1+\frakp_F) = (q_F-1)^{-1}$.
\end{proof}

\begin{lem}\label{lem:r-sint-twist}
$\mathcal{L}(W,M(\tau,s)f_s,\phi)$ equals
\begin{enumerate}
\item[(i)] when $\tauF$ {is unramified},
\[
q_F^{-2s+\frac{1}{2}}\tau(\varpi)(q_F-1)\left(L(2s-1,\tauF)-1-\frac{1}{q_F-1}\right) {\mathcal{L}(W,f_s,\phi)};
\]
\item[(ii)] when $\tauF$ {is ramified},
\[
q_F^{-2s+\frac{1}{2}}\tau(\varpi) G(\psi_F,\tauF){\mathcal{L}(W,f_s,\phi)}.
\]
\end{enumerate}

\end{lem}
\begin{proof}
Applying the definition of the Rankin-Selberg integral $\mathcal{L}(W,M(\tau,s)f_s,\phi)$, the same as formula \eqref{comp:R-Sint} but replacing $f_s$ by $M(\tau,s)f_s$, we have
\begin{equation}\label{aaa}
\vol(\frakp_E)^l\vol(\calO_E)\int_S \psi(c \varpi^{-1}) \Upsilon^{-1}(-a) [M(\tau,s)f_s](b,1) db.
\end{equation}
On $S$, the functions $[M(\tau,s)f_s](b,1)$ is constant when varying $a$, i.e.
\begin{equation*}
[M(\tau,s)f_s](b,1) = [M(\tau,s)f_s]\left(\pmat{1}{}{c}{1},1\right)
\end{equation*}
since $M(\tau,s)f_s \in V(\tau^*,1-s)$. Hence the integral over $S$ in \eqref{aaa} splits into a product of two integrals:
\[
\int_{1+\frakp_E} \Upsilon^{-1}(-a) d^\times a \int_{\frakp_F}\psi(c \varpi^{-1})[M(\tau,s)f_s]\left(\pmat{1}{}{c}{1},1\right) dc.
\]
We split the second integral into the following summation:
\begin{equation}\label{bbb}
\int_{\frakp_F^2}\psi_F(c \varpi^{-1})[M(\tau,s)f_s]\left(\pmat{1}{}{c}{1},1\right) dc + \int_{\frakp_F \setminus \frakp_F^2}\psi_F(c \varpi^{-1})[M(\tau,s)f_s]\left(\pmat{1}{}{c}{1},1\right)dc.
\end{equation}
When $c \in \frakp_F^2$, we note that $\psi(c\varpi^{-1}) = 1$. By Lemma \ref{lem:int.act}, the first part in \eqref{bbb} equals to
\[
A(\tau,s,\psi)\left(L\left(2s-1,\tau|_{F^\times}\right)-1\right)\vol(\frakp_F^2).
\]
where
\[
A(\tau,s,\psi) = \begin{cases} 
q_F^{-2s+\frac{1}{2}} \tau(\varpi) (q_F-1), &\tauF \text{ unramified},\\
0, &\tauF \text{ ramified}.
\end{cases}
\]
When $c \in \frakp_F \setminus \frakp_F^2 = \varpi \calO_F^\times$, we notice that $[M(\tau,s)f_s]\left(\pmat{1}{}{c}{1},1\right) = q_F^{-2s+1}\tau(c)$. The second part in \eqref{bbb} is hence
\begin{align*}
&\int_{\varpi \calO_F^\times} \psi_F(c\varpi^{-1})q_F^{-2s+\frac{1}{2}}\tau(c) dc\\
=& q_F^{-2s+\frac{1}{2}}(q_F^{\frac{1}{2}} - q_F^{-\frac{1}{2}})q_F^{-1} \int_{\varpi \calO_F^\times} \psi_F(c\varpi^{-1})\tau(c) d^\times c\\
=&q_F^{-2s+\frac{1}{2}}(q_F - 1) \vol(\frakp_F^2) \tau(\varpi) \int_{\calO_F^\times} \psi_F(c) \tau(c) d^\times c.
\end{align*}
The last integral is
\begin{align*}
\int_{\calO_F^\times} \psi_F(c) \tau(c) d^\times c &= \sum_{u \in  \calO_F^\times / 1 + \frakp_F} \int_{1 + \frakp_F} \psi_F(uc) \tau(uc) d^\times c \\
&= \sum_{u \in \calO_F^\times / 1 + \frakp_F} \psi_F(u)\tau(u)  \int_{1+ \frakp_F} d^\times c\\
&= (q_F-1)^{-1} G(\psi_F,\tauF).
\end{align*}
In particular, when  $\tauF$ is unramified, $G(\psi_F,\tauF) = -1$. 
\end{proof}

Let $\Upsilon$ and $\psi$ be as denoted in the beginning of this section and let $\gamma^{\text{Sh}}(s,\pi \times \tau, \psi)$ be the gamma factor defined by Shahidi \cite{Sh3}. We now state our main theorem.
\begin{thm}\label{main1}
Let $\pi = \pi_{\omega^1,b}$ and $\tau$ be a tamely ramified character of $E^\times$, then
\[
\gamma(s,\pi \times \tau,\Upsilon, \psi) = \omega^1(-1) \tau(-1)^{l} \tau(b^{-1}\varpi) q_F^{-2s+1},
\]
and equivalently
\[
\gamma^{\text{Sh}}(s,\pi \times \tau, \psi) = -\omega^1(-1) \tau(-1)^{l} \tau(b^{-1}\varpi) q_F^{-2s+1}.
\]

\end{thm}
\begin{proof}
Let $b = 1$. We first note that $\omega^1$ is the central character of $\pi =\pi_{\omega^1,1}$. When $\tauF$ is unramified, $\tau(-1) = 1$. By Lemma \ref{lem:r-sint-twist} and \eqref{defn:norgam}, the Rankin-Selberg gamma factor $\gamma(s,\pi \times \tau, \Upsilon, \psi)$ is 
\[
\omega_\pi(-1) \times \gamma(2s-1,\tau|_{F^\times},\psi_F)  \times q_F^{-2s+\frac{1}{2}}\tau(\varpi)(q_F-1)\left(L(2s-1,\tauF)-1-\frac{1}{q_F-1}\right).
\]
For a level one additive character $\psi$ and unramified $\tauF$ [cf. \cite{BH06} (23.5)],
\[
\epsilon(2s-1,\tauF,\psi_F) = q_F^{2s-\frac{3}{2}} \tau^{-1}(\varpi).
\]
Direct computation shows that
\[
\gamma(s,\pi \times \tau, \psi) = \omega^1(-1) q_F^{-2s+1}\tau(\varpi).
\]

When $\tauF$ is ramified, the Rankin-Selberg gamma factor $\gamma(s,\pi \times \tau, \Upsilon, \psi)$ is
\[
\omega_\pi(-1)\tau(-1)^l \times \tau(-1)\epsilon(2s-1,\tau|_{F^\times},\psi_F) \times q_F^{-2s+\frac{1}{2}}\tau(\varpi)G(\psi,\tauF)
\]
For a level one additive character $\psi$ and tamely ramified ramified $\tauF$, we have [cf. \cite{BH06} (23.6)]
\[
\epsilon(2s-1,\tauF,\psi_F) =  q_F^{-\frac{1}{2}} G(\psi_F,\tauF^{-1}),
\]
and
\[
G(\psi_F,\tauF^{-1})G(\psi_F,\tauF) = \tau(-1) q_F.
\]
Hence the gamma factor equals
\begin{align*}
\omega^1(-1) \tau(-1)^{l} \tau(\varpi) q_F^{-2s+1}.
\end{align*}
To see the formula for an arbitrary $b \in F^\times$, we simply notice that we can vary the uniformizer to $b^{-1} \varpi$ and the above computation will hold identically.

Applying Corollary 1.1 in \cite{Mor23} to $\pi \otimes \Upsilon^{-1}$ we get
\begin{equation*}
\gamma^{\text{Sh}}(s,\pi \times \tau, \psi) = \gamma(s,(\pi \otimes \Upsilon^{-1}) \times \tau, \Upsilon, \psi).
\end{equation*}
By the property of the Rankin-Selberg integral
\[
\gamma(s,(\pi \otimes \Upsilon^{-1}) \times \tau, \Upsilon, \psi) = 
\gamma(s,\pi  \times (\tau \otimes \Upsilon^{-1}), \Upsilon, \psi)
\]
we further have
\begin{equation*}
    \gamma^{\text{Sh}}(s,\pi \times \tau, \psi) = \gamma(s,\pi  \times (\tau \otimes \Upsilon^{-1}), \Upsilon, \psi).
\end{equation*}
Since $\Upsilon(-1) = 1$ and $\Upsilon(\varpi) = -1$, the above equals $-\gamma(s,\pi \times \tau, \Upsilon, \psi)$. To see the other implication, we simply replace $\tau$ with $\tau \otimes \Upsilon$ in the above formula.
\end{proof}

\subsection{Another proof in the non-dyadic case}
In this section we will recover Theorem \ref{main1} in the non-dyadic case using a different method. Indeed, the standard base change endoscopic lift of a simple supercuspidal representation of the unramified unitary group $\UU_{2l}$ is a simple supercuspidal representation of $\GL_{2l}(E)$ (cf. \cite{Oi18}), and the gamma factor of the lift coincides with the original representation (cf. \cite{Mok15,KK05}). Finally, the gamma factor of a simple supercuspidal representation of a general linear group was computed using the Rankin-Selberg method (cf. \cite{AL14}). 

We now explain the above argument in details. First we note that the author used a different Hermitian matrix
$$H = H_{2l}=
\begin{pmatrix}
    &&&&1\\
    &&&-1&\\
    &&\iddots&&\\
    &1&&&\\
    -1&&&&
\end{pmatrix}$$
defining the unitary group in \cite{Oi18}. We start from making a dictionary between the parametrization of the simple supercuspidal representations in the two papers.
For clarification, we denote the simple supercuspidal representation parametrized by the pair $(\omega^1,b)$ in \cite{Oi18} as $\pi^\prime_{\omega^1,b} = \mathrm{cInd}\chi^\prime_{\omega^1,b}$, to be distinguished from our $\pi_{\omega^1,b}$.
We observe the two Hermitian matrices conjugate with each other by a toric element $\eta = \mathrm{diag}((-1)^0,(-1)^1,\ldots,(-1)^{l-1},1,1,\ldots,1)$ and consequently the two unitary groups are isomorphic via the conjugation of the element. The isomorphism preserves the chosen Iwohori subgroups and their filtration groups and hence it suffices to describe the character $\chi^\eta_{\omega^1,b}$ in terms of $\chi^\prime$.
Indeed, we have $$\chi^\eta_{\omega^1,b}(zy) = \omega^1([z])\psi(-y_{12}-\ldots-y_{l-1,l}+\frac{(-1)^{l-1}y_{l,l+1}}{2}+\frac{b\varpi^{-1}y_{2l,1}}{2}).$$
The character 
$\chi^\eta_{\omega^1,b}$ is equivalent to a character in the form that we used to parametrize the simple supercuspidal representations, i.e. for some $\alpha \in T$ there exists $\omega^\prime \in (k_E)^\vee$ and $b^\prime \in F^\times$ such that
\[
(\chi^\eta)^\alpha_{\omega^1,b}=
\chi^\prime_{\omega^\prime,b^\prime}.
\]
Let $\alpha = \mathrm{diag}(z_1,\ldots,z_l,\sigma(z_l)^{-1},\ldots,\sigma(z_1)^{-1})$. By solving the above equation for $\alpha$, we get $\omega^\prime = \omega^1$ and the relations
\[
\begin{cases}
(z_{m-1})^{-1}z_m = -1  \text{ for } m \neq l\\
N(z_1)^{-1} = (-1)^{l-1} \\
N(z_1)b = b^\prime
\end{cases}
\]
where $N=N_{E/F}$ is the norm map with respect to the field extension $E/F$.
Hence $b^\prime = (-1)^{l-1}b$. Putting the above together, $\pi^\prime = \pi^\prime_{\omega^1,(-1)^{l-1}}$ is the corresponding representation under the dictionary. 

Theorem 5.13 in \cite{Oi18} shows that the standard base change lift of the simple supercuspidal representation $\pi^\prime$ is $\Pi = {\pi^{\GL_{2l}(E)}_{\omega,(-1)^{l-1}b,-\omega^{1}(-1)}}$, where the parametrization of simple supercuspidal representations of the general linear group follows that in \cite{Oi18} and where $\omega \in (k_E^\times)^\vee$ is defined as $\omega(z) = \omega^1(z/\sigma(z))$. Shahidi's gamma factor is preserved by the standard base change lift (cf. \cite{Mok15}), i.e. 
\[
\gamma^\text{Sh}(s, \pi \times \tau, \psi) = \gamma(s, \Pi \times \tau, \psi),
\]
where the right hand side is computed in \cite{AL14}, Corollary 3.14 as
\[
\gamma(s,\Pi \times \tau, \psi) = -\omega^1(-1) \tau(-1)^l \tau(b^{-1}\varpi) q_F^{-2s+1}.
\]
Here we note that the lift $\Pi$ is denoted as $\sigma((-1)^{l-1}b^{-1}\varpi, - \omega^1(-1), \tilde\omega)$ in the parametrization in \textit{loc. cit.}, where $\tilde\omega \in (E^\times)^\vee$ is the character which induces $\omega$ on $k_E^\times$. Section A.2 in \cite{AHKO} provided a dictionary between the different parametrizations of the simple supercuspidal representations of the general linear group.

\section{An application}
With our computation of the gamma factor for $p=2$, we aim to provide new insight on the possible endoscopic lift of simple supercuspidal representations from $\UU_{2l}$ to $\GL_{2l}(E)$ and the structure of the $L$-packets of $\UU_{2l}$.
\begin{prop}
Let $p=2$, and let $\pi$ be a simple supercuspidal representation of $\UU_{2l}$. $\pi$ is the unique only element in its $L$-packet.
\end{prop}
\begin{proof}
Let $\psi$ be a fixed generic character of $E$. From \cite{Ato16}, we know that in each $L$-packet $\Pi_\phi$ of $\UU_{2l}$, for a Langlands parameter $\phi$, there exists a representation $\pi_\psi$ that is $\psi$-generic.

Let $\phi$ be the corresponding Langlands parameter of $\pi_\psi$, and let $S_\phi$ denote the cardinality of the $L$-packet containing $\pi_\psi$. Since there are a total of $q^2-1$ simple supercuspidal representations, an odd number when $p=2$, we know that $S_\phi$ is odd.

let $T$ be a maximal torus of $\u_{2l}$. Then there are $F^\times/N_{E/F}(E^\times) = 2$ total $T$-orbits of generic characters of $\UU_{2l}$. As there are two orbits of generic characters, we will call $\psi'$ the other representative of a generic character for $\UU_{2l}$.

Furthermore, note that any simple supercuspidal representation is generic with respect to some generic character, in our case, $\psi$ or $\psi'$.

$S_\phi$ is odd, and if it were greater than one, then we would have at least three simple supercuspidal representations, $\pi_1=\pi_\psi$, $\pi_2$, and $\pi_3$ belonging to the same $L$-packet. Without loss of generality, $\pi_1$ is $\psi$-generic, while $\pi_2$ is $\psi'$-generic. But $\pi_3$ has to be generic with respect to one of these characters. Since there is only one element in each $L$-packet that is $\psi$-generic (resp. $\psi'$-generic) from \cite{Ato16}, we arrive at a contradiction, thus showing that $S_\phi = 1$.
\end{proof}

In understanding the structure of the Langlands parameters for simple supercuspidal representations of $\UU_{2l}$, we suspect the course of action is to adapt Section 6.2 and 6.3 of \cite{AHKO} to this setting. If we wish to conclude that endoscopic lifts for $p=2$ are consistent with those of Oi (cf. \cite{Oi18}), we should show that $\phi$ the parameter corresponding to $\pi_{\omega^1,1}$ is irreducible and that the Swan conductor of the $\phi$ is $1$.

Because we have no poles of the gamma factor $\gamma^{\text{Sh}}(s,\pi \times \tau, \psi)$ at $s=1$, we expect $\phi$ to be irreducible. Our claim for irreducibility may be achievable using the formal degree conjecture for $\u_{2l}$ (cf., Section 6.3 \cite{AHKO}), which was recently proven by Beuzart-Plessis \cite{BP21}. 

Following the equation for the Swan conductor in  Corollary 6.6 of \cite{AHKO}, and assuming this equation is still valid for quasi-split groups, we find that in our case of $p \neq 2$ as well as in the case of $p=2$, we have 
\[
\vert \epsilon(s,\phi,\psi) \vert = \vert \gamma(s,\phi,\psi) \vert = q_E^{\text{Swan}(\phi)(\frac{1}{2}-s)}
\]
and
\[
\vert \gamma(s,\phi,\psi) \vert = \left\vert 
-\omega^1(-1) \tau(-1)^{l} \tau(\varpi) q_F^{-2s+1} \right\vert = q_E^{\frac{1}{2}-s}, 
\]
we would conclude that the $\text{Swan}(\phi) = 1$. This would thus prove the lift of $\pi$ is to a simple supercuspidal of $\GL_{2l}(E)$. And since $S_\phi = 1$, we would further know that no two simple supercuspidal representations lift to the same representation of $\GL_{2l}(E)$.



\end{document}